\documentclass[12pt,reqno]{amsart}
\usepackage{amsfonts}
\usepackage{amsmath}
\usepackage{amsthm}
\usepackage{mathrsfs}
\usepackage{hyperref}
\usepackage{textcomp}
\usepackage{mathrsfs,stmaryrd}

\newtheorem{theorem}{Theorem}
\newtheorem{lemma}{Lemma}
\newcommand{\Tr}{{\rm Tr\,}}
\theoremstyle{definition}
\newtheorem{definition}{Definition}
\newtheorem{assumption}{Assumption}
\begin{document}
\title[Stochastic Perron's method]{Stochastic Perron's method for optimal control problems with state constraints}

\author{Dmitry B. Rokhlin}

\address{D.B. Rokhlin,
Institute of Mathematics, Mechanics and Computer Sciences,
              Southern Federal University,
Mil'chakova str., 8a, 344090, Rostov-on-Don, Russia}
\email{rokhlin@math.rsu.ru}

\begin{abstract} We apply the stochastic Perron method of Bayraktar and S\^irbu to a general infinite horizon optimal control problem, where the state $X$ is a controlled diffusion process, and the state constraint is described by a closed set. We prove that the value function $v$ is bounded from below (resp., from above) by a viscosity supersolution (resp., subsolution) of the related state constrained problem for the Hamilton-Jacobi-Bellman equation. In the case of a smooth domain, under some additional assumptions, these estimates allow to identify $v$ with a unique continuous constrained viscosity solution of this equation.
\end{abstract}
\subjclass[2010]{93E20, 49L25, 60H30}
\keywords{Stochastic Perron's method, state constraints, viscosity solution, comparison result}

\maketitle

\section{Introduction and the main result}
\label{sec:1}
The aim of the paper is to extend the scope of applications of the stochastic Perron method, developed by Bayraktar and S\^irbu. This method allows to characterise the value function of a controlled diffusion problem as a viscosity solution of the corresponding Hamilton-Jacobi-Bellman (HJB) equation, bypassing the dynamic programming principle. Instead it requires a comparison result, implying the uniqueness of a viscosity solution of the HJB equation. Previously this method was applied to linear parabolic equations \cite{BaySir12}, stochastic differential games \cite{BaySir14,Sir13,Sir14}, regular \cite{BaySir13,Rok13} and singular control problems \cite{BayZha14}.

The method involves the construction of two families $\mathcal V_-$, $\mathcal V_+$ of functions, bounding the value function from below and above
\[ u\le v\le w,\ \ \ u\in \mathcal V_-,\ \ w\in\mathcal V_+. \]
Elements of $\mathcal V_-$, $\mathcal V_+$ are called stochastic sub- and supersolutions. By the superposition with the state process, $u$ and $w$ generate sub- and supermartingale-like processes. Similarly to the classical Perron method \cite[Sections 2.8, 6.3]{GilTru01}, the set $\mathcal V_-$ (resp., $\mathcal V_+$) is directed upward (resp., downward) with respect to the pointwise maximum (resp., minimum) operation. The essence of the method is to prove that the functions
\[ u_-(x)=\sup_{u\in\mathcal V_-} u(x),\ \ \ w_+(x)=\inf_{w\in\mathcal V_+} w(x) \]
are respectively viscosity super- and subsolutions of the related HJB equation. If a comparison result, providing the inequality $u_-\ge w_+$, holds true, it follows that $u_-=v=w_+$ is a unique (continuous) viscosity solution. This construction differs from Perron's method of \cite{Ish87}, which is not linked to the value function.

In the present paper we consider the stochastic control problem with state constraints in the form of \cite{Kats94}. In contrast to \cite{LasLio89}, where the drift is not assumed to be bounded, and the value function is singular near the boundary, in \cite{Kats94} the problem is "regular". To achieve the regularity it is assumed that the diffusion coefficient depends on the control and degenerates at the boundary. The same problem was considered in \cite{IshLor02,BucGorQui11}. It was proved that under appropriate assumptions the value function $v$ is a unique continuous constrained viscosity solution of the HJB equation. (The term "constrained" means, in particular, that $v$ satisfies special boundary conditions, which in the deterministic situation were introduced in \cite{Son86}.) Roughly speaking, it is enough to assume that for each boundary point there exists a control, which kills the diffusion and directs the drift strictly inside the domain.

An application of the stochastic Perron method to state constrained problems seems rather interesting, since, as it is mentioned in \cite{Kats94}, a direct proof of the dynamic programming principle is not available due to a complicated structure of admissible control processes, retaining a phase trajectory in a predetermined domain. Different penalization and approximation procedures were used instead in \cite{Kats94,IshLor02,BucGorQui11,BouNut12}.

We turn to the precise statement of our main result (Theorem \ref{th:1}). Let $\Omega$ be the space $C([0,\infty),\mathbb R^m)$ of continuous $\mathbb R^m$-valued functions, endowed with the $\sigma$-algebra $\mathscr F^\circ$ of cylindrical sets, and let $\mathsf P$ be the Wiener measure on $\mathscr F^\circ$. So, the canonical process $W_s(\omega)=\omega(s)$ is the standard $m$-dimensional Brownian motion under $\mathsf P$. Denote by $\mathbb F^\circ = (\mathscr F_t^\circ)_{t\ge 0
}$ the natural filtration of $W$, and let $\mathbb F = (\mathscr F_t)_{t\ge 0
}$ be the correspondent minimal augmented filtration. The extension of the Wiener measure to the completion $\mathscr F$ of $\mathscr F^\circ$ is still denoted by $\mathsf P$.

Let $\alpha$ be an $\mathbb F$-progressively measurable stochastic process with values in a compact set $A\subset\mathbb R^k$, $0\in A$. Consider the system of stochastic differential equations
\begin{equation} \label{eq:1.1}
 dX_t=b(X_t,\alpha_t) dt+\sigma(X_t,\alpha_t) dW_t,\ \ X_0=x.
\end{equation}
We assume that the drift vector $b:\mathbb R^d\times A\mapsto\mathbb R^d$ and the diffusion matrix $\sigma:\mathbb R^d\times A\mapsto\mathbb R^d\times\mathbb R^m$ are continuous and satisfy the Lipschitz condition
\[ |b(x,a)-b(y,a)|+|\sigma(x,a)-\sigma(y,a)| \le K |x-y| \]
with some constant $K$ independent of $x$, $y$, $a$. Note, that the linear growth condition
\[ |b(x,a)|+|\sigma(x,a)| \le K' (1+|x|) \]
follows from the continuity of $b$, $\sigma$ and compactness of $A$. Thus, there exist a unique $\mathbb F$-adapted strong solution  $X^{x,\alpha}$ of (\ref{eq:1.1}) on $[0,\infty)$: see \cite[Chapter 2, Sect. 5]{Kry80}.

Let $G\subset\mathbb R^d$ be a closed set with the boundary $\partial G$ and nonempty interior $G^\circ$. It will be convenient to assume that $0\in G^\circ$. Denote by $\mathscr A(x)$, $x\in G$ the set of $\mathbb F$-progressively measurable control processes $\alpha$ with values in $A$ and such that $X_t^{x,\alpha}\in G$, $t\ge 0$ a.s. Elements of $\mathscr A(x)$ are called \emph{admissible controls} for the initial condition $x$. The cost functional $J$ and the value function $v$ are defined as follows
\begin{equation} \label{eq:1.2}
J(x,\alpha)=\mathsf E\int_0^\infty e^{-\beta s} f(X_s^{x,\alpha},\alpha_s)\,ds, \ \ \ v(x)=\inf_{\alpha\in\mathscr A(x)} J(x,\alpha),
\end{equation}
where $f:G\times A\mapsto\mathbb R$ is a bounded continuous function.

We assume that  for any initial condition $x\in G$ there exists an admissible control: $\mathscr A(x)\neq\emptyset$. In this case the set $G$ is called \emph{viable}. A necessary condition for the validity of this property is given in \cite{BarJen02} (Theorem 1). Let
\[ \mathscr N_G^2(x)=\left\{(p,Y)\in\mathbb R^d\times\mathbb S^d:\liminf_{G\ni y\to x}\left(\frac{p\cdot(y-x)}{|y-x|^2}+\frac{1}{2}\frac{Y(y-x)\cdot(y-x)}{|y-x|^2}\right)\ge 0\right\} \]
be the second order normal cone. Here $\mathbb S^d$ is the set of symmetric $d\times d$ matrices. If the set $G$ is viable then for all $x\in\partial G$, $(p,Y)\in\mathscr N_G^2(x)$ there exist $a\in A$ such that
\begin{equation} \label{eq:1.3}
 p\cdot b(x,a)+\frac{1}{2}\Tr(\sigma(x,a)\sigma^T(x,a)Y)\ge 0.
\end{equation}
See \cite[Section 3]{BarJen02} for more concrete forms of this condition.

We impose a slightly stronger requirement. For any function $\psi:\mathbb R^d\mapsto A$ put
\begin{equation} \label{eq:1.4}
b_\psi(x)=b(x,\psi(x)),\ \ \ \sigma_\psi(x)=\sigma(x,\psi(x)).
\end{equation}
\begin{assumption} \label{as:1}
There exist a Borel measurable function $\psi:\mathbb R^d\mapsto A$ such that $b_\psi$, $\sigma_\psi$ are globally Lipschitz continuous and
\[ p\cdot b_\psi(x)+\frac{1}{2}\Tr(\sigma_\psi(x)\sigma^T_\psi(x) Y)\ge 0, \ \ x\in\partial G, \ (p,Y)\in\mathscr N_G^2(x).\]
\end{assumption}
Under this assumption there exist a unique strong solution of the equation
\begin{equation} \label{eq:1.5}
 dX_t=b_\psi(X_t) dt+\sigma_\psi(X_t) dW_t,\ \ X_0=x
\end{equation}
and $X_t\in G$, $t\ge 0$ a.s.: see \cite[Theorem 3.1]{BarGoa99}. The correspondent control process $\alpha_t=\psi(X_t)$ is admissible for $x$. Hence, $\mathscr A(x)\neq\emptyset$, $x\in G$.

Consider the Bellman operator
\[ F(x,r,p,Y)=\sup_{a\in A}\left(\beta r-f(x,a)-b(x,a)\cdot p-\frac{1}{2}\Tr(\sigma(x,a)\sigma^T(x,a)Y)\right), \]
defined on $\mathbb R\times\mathbb R\times\mathbb R^d\times\mathbb S^d$.
Recall that a bounded upper semicontinuous (usc) function $u$ is called a \emph{viscosity subsolution} of the equation
\begin{equation} \label{eq:1.6}
 F(x,u,D u, D^2 u)=0
\end{equation}
on a set $E\subset\mathbb R^d$ if for any $\varphi\in C^2(\mathbb R^d)$ and for any local maximum point $x_0$ of $u-\varphi$ on $E$ the inequality
\[ F(x_0,u(x_0),D\varphi(x_0),D^2\varphi(x_0))\le 0 \]
holds true. In the same way, a bounded lower semicontinuous (lsc) function $w$ is called a \emph{viscosity supersolution} of (\ref{eq:1.6}) on $E$ if for any $\varphi\in C^2(\mathbb R^d)$ and for any local minimum point $x_0$ of $w-\varphi$ on $E$ we have the inequality
\[ F(x_0,w(x_0),D\varphi(x_0),D^2\varphi(x_0))\ge 0. \]

In these definitions one can assume that the maximum (resp., minimum) point $x_0$ is strict and $\varphi(x_0)=u(x_0)$ (resp., $\varphi(x_0)=w(x_0)$).

It is convenient to introduce the \emph{state constrained} problem
\begin{equation} \label{eq:1.7}
 \left\{\begin{array}{l l}
    F(x,u,Du,D^2 u)\le 0 & \text{on}\quad G^\circ,\\
    F(x,u,Du,D^2 u)\ge 0 & \text{on}\quad G.
  \end{array} \right.
\end{equation}
We say that a bounded usc (resp., lsc) function $u$, defined on $G$, is viscosity subsolution (resp., supersolution) of the state constrained problem (\ref{eq:1.7}) if $F(x,u,Du,D^2 u)\le 0$ on $G^\circ$ (resp.,  $F(x,u,Du,D^2 u)\ge 0$ on $G$) in the viscosity sense. A bounded function $u$ is called a viscosity solution of (\ref{eq:1.7}) (or a \emph{constrained viscosity solution}), if its upper semicontinuous envelope $u^*$ is a viscosity subsolution, and its lower semicontinuous envelope $u_*$ is a viscosity supersolution of (\ref{eq:1.7}).

Denote by $\Gamma$ the set of points $x\in\partial G$ such that for some $\alpha\in\mathscr A(x)$ the solution $X^{x,\alpha}$ of (\ref{eq:1.1}) immediately enters $G^\circ$ with probability $1$:
\[ \mathsf P(\inf\{t>0:X_t^{x,\alpha}\in G^\circ\}=0)=1. \]

\begin{theorem} \label{th:1}
There exist a viscosity subsolution $w_+$ and a viscosity supersolution $u_-$ of the state constrained problem (\ref{eq:1.7}) such that
\[ u_-\le v\ \quad\text{on}\quad G;\quad v\le w_+\quad \text{on}\quad G^\circ, \]
and $v(x)\le\limsup_{G^\circ\ni y\to x} w_+(y)$, $x\in\Gamma.$
\end{theorem}

The nature of  $w_+$ and $u_-$ is not explicitly indicated  here. Their construction, which is presented in Sections \ref{sec:2} and \ref{sec:3} respectively, is based on the technique of stochastic semisolutions, developed in \cite{BaySir12,BaySir13,BaySir14}. The details are quite similar to \cite{BaySir13,Rok13}. One only should take care of admissibility of controls.

Theorem \ref{th:1} is useful if a sort of comparison result is available, and one can conclude that $w_+\le u_-$.
In Section \ref{sec:4} we consider the case of a smooth domain and, under some additional assumptions, mention that such inequality follows from the known result, concerning the boundary behavior of viscosity subsolutions of linear equations \cite{BarRou98}, and the comparison result of \cite{Kats94}. In combination with Theorem \ref{th:1} this allows to identify $v$ with a unique continuous viscosity solution of (\ref{eq:1.7}). The related result (Theorem \ref{th:2}) is not new and is presented only to demonstrate the capabilities of the stochastic Perron method.

\section{Stochastic supersolutions}
\label{sec:2}
\setcounter{equation}{0}
For  $\mathbb F$-stopping times $\tau$, $\sigma$ and a set $D\in\mathcal F_\tau$ denote by
\[ \llbracket\tau,\sigma\rrbracket = \{(t,\omega)\in [0,\infty) \times\Omega: \tau(\omega)\le t\le\sigma(\omega)\} \]
the stochastic interval, and by
\[ \tau_D = \tau I_D+(+\infty)I_{D^c},\ \ D^c=\Omega\backslash D \]
the restriction of $\tau$ on $D$. Put $B_\varepsilon (x)=\{y\in\mathbb R^d:|y-x|<\varepsilon\}$ and denote by $\overline B_\varepsilon (x)$ the closure of this ball.

Let $\tau:\Omega\mapsto [0,\infty]$ be a stopping time and take an $\mathscr F_\tau$-measurable random vector $\xi$ such that $\xi I_{\{\tau<\infty\}}$ is bounded and $\xi\in G$ on $\{\tau<\infty\}$. For an $\mathbb F$-progressively measurable process $\alpha$ with values in $A$ consider the stochastic differential equation (\ref{eq:1.1}) with the \emph{randomized initial condition} $(\tau,\xi)$:
\begin{equation} \label{eq:2.1}
X_t=\xi I_{\{t\ge\tau\}}+\int_\tau^t b(X_s,\alpha_s)\,ds+\int_\tau^t\sigma(X_s,\alpha_s)\,dW_s,\ \ t\ge 0.
\end{equation}
By $\int_\tau^t(\cdot)$ we mean $\int_0^t I_{\{s\ge\tau\}}(\cdot)$. As is known, see \cite[Chapter 2, Sect.\,5]{Kry80}, there exists a pathwise unique strong solution $X^{\tau,\xi,\alpha}$ of (\ref{eq:2.1}). The trajectories of the process $X^{\tau,\xi,\alpha}$ are continuous on the stochastic interval $\llbracket\tau,\infty\rrbracket$. Moreover, $X^{\tau,\xi,\alpha}=0$ on $\llbracket 0,\tau\llbracket$ and
\[ X_\tau^{\tau,\xi,\alpha}=\lim_{t\searrow\tau} X_t^{\tau,\xi,\alpha}=\xi \quad\text{on}\quad \{\tau<\infty\}. \]

Denote by $\mathscr A(\tau,\xi)$ the set of progressively measurable control processes $\alpha$ such that $\alpha_t\in A$ and $X^{\tau,\xi,\alpha}_t\in G$, $t\in[\tau,\infty)$ a.s. That is, $\mathscr A(\tau,\xi)$ is the set of admissible controls for a randomized initial condition $(\tau,\xi)$. We omit index $\tau$ if $\tau=0$. For instance, $X^{x,\alpha}=X^{0,x,\alpha}$, $\mathscr A(x)=\mathscr A(0,x)$.
\begin{lemma}
Under Assumption \ref{as:1} the set $\mathscr A(\tau,\xi)$ is non-empty for any randomized initial condition $(\tau,\xi)$.
\end{lemma}
\begin{proof}

For an $\mathbb F^\circ$-stopping time $\tau'$ the $\sigma$-algebra $\mathscr F^\circ_{\tau'}$ is countably generated (\cite[Lemma 1.3.3]{StrVar79}), and there exists a regular conditional probability distribution  $\mathsf P^{\tau'}=(\mathsf P^{\tau',\omega})_{\omega\in\Omega}$ of $\mathsf P$ with respect to $\mathscr F^\circ_{\tau'}$:
see \cite[Theorem 1.3.4]{StrVar79} or \cite[Theorem 9.2.1]{Str10}. For each $B\in\mathscr F^\circ$ the function $\omega\mapsto\mathsf P^{{\tau'},\omega}(B)$ is $\mathscr F^\circ_{\tau'}$-measurable, for each $\omega\in\Omega$ the function $B\mapsto\mathsf P^{\tau',\omega}(B)$ is a probability measure on $\mathscr F^\circ$ such that
\[ \mathsf P^{\tau',\omega}(B)=\mathsf E(I_B|\mathscr F^\circ_{\tau'})(\omega)\ \ \mathsf P\mbox{-a.s.},\ \ B\in\mathscr F^\circ.\]
Moreover, there exists a $\mathsf P$-null set $N\in\mathscr F^\circ_{\tau'}$ with the property that
\begin{equation} \label{eq:2.1a}
 \mathsf P^{\tau',\omega}(C)=I_C(\omega)\qquad \text{for all}\quad \omega\not\in N,\ C\in\mathscr F^\circ_{\tau'}.
\end{equation}

Consider the SDE
\begin{equation} \label{eq:2.2}
X_t=\xi I_{\{t\ge\tau\}}+\int_\tau^t b_\psi (X_s)\,ds+\int_\tau^t\sigma_\psi(X_s)\,dW_s,\ \ t\ge 0,
\end{equation}
where $\psi$ satisfies Assumption \ref{as:1}.
To work with $\mathsf P^{\tau'}$, related to the raw filtration $\mathbb F^\circ$, we pass from $\xi I_{\{t\ge\tau\}}$ to an indistinguishable $\mathbb F^\circ$-adapted process of the same form. Recall that any $\mathbb F$-stopping time is predictable (see \cite[Proposition 16.22]{Bass11}) and the filtration $\mathbb F$ is quasi-left continuous (see \cite[Theorem 3.40]{HeWangYan92}), that is, $\mathscr F_{\tau-}=\mathscr F_\tau$ for any (predictable) $\mathbb F$-stopping time $\tau$. By Theorem IV.78 of \cite{DelMey78} there exists an $\mathbb F^\circ$ stopping time $\tau'$ such that $\mathsf P(\tau'\neq\tau)=0$, and for any $B\in \mathscr F_{\tau-}=\mathscr F_\tau$ there exists $B'\in\mathscr F^\circ_\tau$ such that $\mathsf P(I_{B'}\neq I_B)=0$. It easily follows that the process $\xi I_{\{t\ge\tau\}}$ is indistinguishable from an $\mathbb F^\circ$-adapted process $\xi' I_{\{t\ge\tau'\}}$ with some $\mathscr F^\circ_{\tau'}$-measurable $\xi'$.

Put $Z^0_t=t-t\wedge\tau$, $Z_t=W_t-W_{t\wedge\tau}$. The process $Z$ is a continuous martingale under $\mathsf P$, and we can rewrite equation (\ref{eq:2.2}) in the form
\begin{equation} \label{eq:2.2a}
X_t=H_t+\int_0^t b_\psi (X_s)\,dZ^0_s+\int_0^t\sigma_\psi (X_s)\,dZ_s,\ \ t\ge 0,
\end{equation}
where $H_t=\xi' I_{\{t\ge\tau'\}}$.

Recall the pathwise construction of a strong solution, presented in \cite{Kar95} (see also \cite{Bich81,Kar81}). Denote by $\mathbb D=\mathbb D([0,\infty),\mathbb R^d)$ the set of functions from $[0,\infty)$ to $\mathbb R^d$, which are right continuous and have left limits. There exist a mapping $\mathscr S:\mathbb D\times C([0,\infty),\mathbb R^m)\mapsto \mathbb D$ such that if $Z$ is a continuous semimartingale on a filtered probability space $(\Omega,\overline{\mathscr F},\mathsf Q,\overline{\mathbb F})$, where $\overline{\mathbb F}$ satisfies the usual conditions, and if $H$ is an $\overline{\mathbb F}$-adapted process with trajectories in $\mathbb D$, then
\[ \overline X_t(\omega)=\mathscr S (H_\cdot(\omega),Z_\cdot(\omega))_t \]
is a strong solution of (\ref{eq:2.2a}).

Take $\overline\omega\in\Omega\backslash N$ with $\tau'(\overline\omega)<\infty$. Note that $Z$ is a $\mathsf P^{\tau',\overline\omega}$-martingale, and $Z$ is the standard $d$-dimensional $\mathsf P^{\tau',\overline\omega}$-Brownian motion on $[\tau'(\overline\omega),\infty)$. It follows that $\overline X$ is a strong solution of (\ref{eq:2.2a}) under $\mathsf P^{\tau',\overline\omega}$ with respect to the $\mathsf P^{\tau',\overline\omega}$-augmentation of $\mathbb F^\circ$. Moreover, by (\ref{eq:2.1a}) we get
$$ \mathsf P^{\tau',\overline\omega}\left(\{\omega:\tau'(\omega)=\tau'(\overline\omega),\ \xi'(\omega)=\xi'(\overline\omega)\}\right)=1.$$
Hence, under $\mathsf P^{\tau',\overline\omega}$, the process $H$ is indistinguishable from $\xi'(\overline\omega) I_{\{t\ge\tau'(\overline\omega)\}}$, and $\overline X$  is a strong solution of the SDE with a non-random initial condition:
\[ \overline X_t=\xi'(\overline\omega)+\int_{\tau'(\overline\omega)}^t b_\psi(\overline X_s)\,ds+\int_{\tau'(\overline\omega)}^t\sigma_\psi(\overline X_s)\,dW_s,\ \ t\ge \tau'(\overline\omega). \]
In addition, $\overline X_t=0$, $t\in [0,\tau'(\overline\omega))$ $\mathsf P^{\tau',\overline\omega}$-a.s. since $Z_0$, $Z$, $H$ are indistinguishable from $0$ on $[0,\tau'(\overline\omega))$.

By Assumption \ref{as:1} the diffusion coefficients $b_\psi$, $\sigma_\psi$ satisfy conditions of Theorem 3.1 of \cite{BarGoa99}. Since $0\in G$ and $\xi'(\overline\omega)\in G$, we conclude that $\overline X_t\in G$, $t\ge 0$  $\mathsf P^{\tau',\overline\omega}$-a.s. It follows that $G$ is invariant under $\mathsf P$:
\[ \mathsf P(\overline X_t\in G,\ t\ge 0)=\mathsf E\left(I_{\{\tau'(\omega)<\infty\}}\mathsf P^{\tau',\omega}(\overline X_t\in G,\ t\ge 0)\right)=1. \]
The desired control process  $\alpha\in\mathscr A(\tau,\xi)$ is given by the formula $\alpha=\psi(\overline X)$.
\end{proof}

Let $w$ be a uniformly bounded continuous function: $w\in C_b(G)$. Consider the stochastic process
\[ Z_t^{\tau,\xi,\alpha}(w)=\int_\tau^t e^{-\beta s} f(X_s^{\tau,\xi,\alpha},\alpha_s)\,ds+I_{\{t\ge\tau\}}e^{-\beta t} w(X_t^{\tau,\xi,\alpha}). \]

\begin{definition} \label{def:1}
We say that a control process $\alpha\in\mathcal A(\tau,\xi)$ is $w$-\emph{suitable} for $(\tau,\xi)$ if	
\[ \mathsf E(Z_\rho^{\tau,\xi,\alpha}(w)|\mathscr F_\tau)\le Z_\tau^{\tau,\xi,\alpha}(w)=e^{-\beta\tau} w(\xi) \]
for any stopping time $\rho\ge\tau$. A function $w\in C_b(G)$ is called a \emph{stochastic supersolution} of (\ref{eq:1.7}) if for any randomized initial condition $(\tau,\xi)$ with $\xi\in G^\circ$ there exists a $w$-suitable control $\alpha$.
\end{definition}

The set of stochastic supersolutions is denoted by $\mathcal V^{+}$.
Note that in the above definition the values $X_\infty$ are irrelevant, since $Z_\infty=\int_0^\infty e^{-\beta s} f(X_s^{\tau,\xi,\alpha},\alpha_s)\,ds$. We emphasize also that the condition $\mathscr A(\tau,\xi)\neq\emptyset$ for all randomized initial conditions $(\tau,\xi)$, $\xi\in G^\circ$ is necessary for the existence of stochastic supersolutions.

A stochastic supersolution $w$ is an upper bound for the value function (\ref{eq:1.2}) on $G^\circ$. To see this put $\tau=0$, $\xi=x\in G^\circ$, $\rho=\infty$ and take a $w$-suitable control $\alpha\in\mathscr A(x)$. By Definition \ref{def:1}, with the convention $Z^{x,\alpha}=Z^{0,x,\alpha}$, we get
\[ v(x)\le J(x,\alpha)=\mathsf E Z_\infty^{x,\alpha}(w)\le\mathsf E Z_0^{x,\alpha}(w)=w(x).\]

The set $\mathcal V^+$ is non-empty and contains sufficiently large constants $c$: it is easy to see that
\[ \mathsf E(Z_\rho^{\tau,\xi,\alpha}(c)|\mathscr F_\tau)\le c e^{-\beta\tau}=Z_\tau^{\tau,\xi,\alpha}(c) \quad\text{for}\quad c\ge\overline f/\beta,\]
where $\overline f=\sup_{(x,a)\in G\times A} f(x,a).$

\begin{lemma} \label{lem:2}
If $w_1$, $w_2$ are stochastic supersolutions then $w=w_1\wedge w_2$ is a stochastic supersolution.
\end{lemma}
\begin{proof} Let $\alpha^i\in\mathscr A(\tau,\xi)$, $i=1,2$ be $w_i$-suitable controls for a randomized initial condition $(\tau,\xi)$. Put $A_1=\{w_1(\xi)<w_2(\xi)\}\in\mathscr F_\tau$, $A_2=A_1^c:=\Omega\backslash A_1$. We claim that
\[ \alpha=I_{A_1} I_{\{\tau\le t\}}\alpha^1+I_{A_2} I_{\{\tau\le t\}}\alpha^2\]
belongs to $\mathscr A(\tau,\xi)$ and that it is $w$-suitable.

The process $Y=\sum_{i=1}^2 X_t^{\tau,\xi,\alpha^i}I_{A_i}$ satisfy the same equation as $X^{\tau,\xi,\alpha}$. From the pathwise uniqueness property it follows that $Y=X^{\tau,\xi,\alpha}$. We have $X^{\tau,\xi,\alpha}\in G$, $t\ge\tau$ $\mathsf P$-a.s., and
$\alpha$ is $w$-suitable for $(\tau,\xi)$:
\begin{align*}
\mathsf E(Z_\rho^{\tau,\xi,\alpha}(w)|\mathscr F_\tau) & = \sum_{i=1}^2 \mathsf E(I_{A_i} Z_\rho^{\tau,\xi,\alpha^i}(w)|\mathscr F_\tau)\le \sum_{i=1}^2 I_{A_i} \mathsf E(Z_\rho^{\tau,\xi,\alpha^i}(w_i)|\mathscr F_\tau)\\
&\le \sum_{i=1}^2 I_{A_i}e^{-\beta\tau} w_i(\xi) =e^{-\beta\tau} w(\xi). \qedhere
\end{align*}
\end{proof}

The following result was used in \cite{BaySir12,BaySir13,Rok13} (see, e.g., Lemmas 2 and 4 of \cite{Rok13}). Its proof use only the fact that $\mathcal V^+$ is directed downward, that is, the statement of Lemma \ref{lem:2} holds true.
\begin{lemma} \label{lem:3}
There exists a sequence $w_n\in\mathcal V^+$, $w_n(x)\ge w_{n+1}(x)$, $x\in G$ such that
\[\lim_{n\to\infty} w_n(x)=w_+(x):=\inf\limits_{u\in\mathcal V^+} w(x).\]
\end{lemma}

The next assertion is the most important part of the stochastic Perron method.
\begin{lemma} \label{lem:4}
The function
\[ w_+(x)=\inf\limits_{w\in\mathcal V^+} w(x)\]
is a viscosity subsolution of (\ref{eq:1.7}).
\end{lemma}
\begin{proof}
If $w_+$ is not a viscosity subsolution then there exist $x_0\in G^\circ$, $\varphi\in C^2$ and $\varepsilon>0$ such that $w_+(x_0)=\varphi(x_0)$, $w_+<\varphi$ on the set $\overline B_\varepsilon (x_0) \backslash\{0\}\subset G^\circ$  and
\[ F(x_0,\varphi(x_0),D\varphi(x_0),D^2\varphi(x_0))> 0. \]

Hence, there exists some $a\in A$ such that $\beta\varphi(x_0)-(\mathcal L^a\varphi)(x_0)-f(x_0,a)>0$, where
\[ (\mathcal L^a\varphi)(x)=b(x,a)D\varphi(x)+\frac{1}{2}\Tr\left(\sigma(x,a)\sigma^T(x,a)D^2\varphi(x)\right).\]
 By the continuity of $b$, $\sigma$, $f$ we may assume that
\begin{equation} \label{eq:2.3}
\beta\varphi(x)-(\mathcal L^a\varphi)(x)-f(x,a)>0,\ \ \ x\in\overline B_\varepsilon(x_0)\subset G^\circ
\end{equation}
for some $\varepsilon>0$.

Since $w_+$ is upper semicontinuous, we have
\[ w_+(x)-\varphi(x)\le-\delta<0,\ \ x\in S_\varepsilon:=\overline B_\varepsilon(x_0)\backslash B_{\varepsilon/2}(x_0).\]
By Lemma \ref{lem:3} there exists a decreasing sequence $w_n\in\mathcal V^+$, $w_n\searrow w_+$. The sets
\[ A_n=\{x\in S_\varepsilon:w_n(x)-\varphi(x)\ge-\delta'\},\ \ \delta'\in (0,\delta) \]
are compact, $A_n\supset A_{n+1}$ and
$\cap_{n=1}^\infty A_n=\emptyset$. Thus, $\cap_{n=1}^N A_n=\emptyset$ for some $N$. This means that there exists a function $w=w_N\in\mathcal V^+$ such that $w-\varphi<-\delta'$ on $S_\varepsilon$.

Define the function $\varphi^\eta=\varphi-\eta$, where $\eta\in (0,\delta')$ is such that the inequality (\ref{eq:2.3}) holds true for $\varphi^\eta$ instead of $\varphi$. Note that
\[ w-\varphi^\eta=w-\varphi+\eta<-\delta'+\eta<0 \quad\text{on}\quad S_\varepsilon.\]
We claim that
\[ w^\eta=\left\{\begin{array}{l l}
    \varphi^\eta\wedge w  & \text{on}\quad B_\varepsilon(x_0),\\
    w & \ \text{otherwise}
  \end{array} \right. \]
is a stochastic subsolution. This gives a contradiction with the definition of $w_+$ since $w^\eta(x_0)=\varphi^\eta(x_0)=w_+(x_0)-\eta < w_+(x_0)$.

It is clear that $w^\eta\in C_b(G)$. We only need to construct a $w^\eta$-suitable control $\alpha$ for a randomized initial condition $(\tau,\xi)$, $\xi\in G^\circ$. Put
\[ U=\{x\in B_{\varepsilon/2}(x_0):w(x)>\varphi^\eta(x)\},\ \ H=\{\xi\in U\}\in\mathscr F_\tau \]
and define a progressively measurable process
\[ \overline\alpha_t=(a I_H+\alpha^0_t I_{H^c})I_{\{t\ge\tau\}}\in A,\]
where $\alpha^0$ is a $w$-suitable control for $(\tau,\xi)$. Furthermore, put
\begin{align*}
\tau_1 &= \inf\{t\ge\tau:X^{\tau,\xi,\overline\alpha}_t\not\in B_{\varepsilon/2}(x_0)\},\\
\alpha_t &= \overline\alpha_t I_{\{t\le\tau_1\}} + \alpha^1_t I_{\{t>\tau_1\}},
\end{align*}
where $\alpha^1$ is a $w$-suitable control for $(\tau_1,\xi_1)$, $\xi_1=X^{\tau,\xi,\overline\alpha}_{\tau_1} I_{\{\tau_1<\infty\}}$. We have $X^{\tau,\xi,\alpha}=X^{\tau,\xi,\overline\alpha}$ on the stochastic interval $\llbracket\tau,\tau_1\rrbracket$ and $X^{\tau,\xi,\alpha}=X^{\tau_1,\xi_1,\alpha^1}$ on $\llbracket\tau_1,\infty\rrbracket$.
Thus, $\alpha\in\mathscr A(\tau,\xi)$. Note also that for $E=\{\xi\in B_{\varepsilon/2}(x_0)\}$ we get
\[ X^{\tau,\xi,\alpha}\in \overline B_{\varepsilon/2}(x_0) \quad\text{on}\quad \llbracket\tau_E,(\tau_1)_E\rrbracket;\ \
    X^{\tau,\xi,\alpha}=\xi \quad\text{on}\quad \llbracket\tau_{E^c},(\tau_1)_{E^c}\rrbracket.\]

It remains to show that $\alpha$ is a $w^\eta$-suitable control for $(\tau,\xi)$. For a stopping time $\rho\ge\tau$ put $D=\{\rho>\tau_1\}$.  We have
\begin{align} \label{eq:2.4}
Z^{\tau,\xi,\alpha}_\rho(w^\eta) I_D &=I_D\int_\tau^{\tau_1} e^{-\beta s} f(X_s^{\tau,\xi,\overline\alpha},\overline\alpha_s)\,ds\nonumber\\
&+ I_D \left(\int_{\tau_1}^\rho e^{-\beta s} f(X_s^{\tau_1,\xi_1,\alpha^1},\alpha^1_s)\,ds+ e^{-\beta\rho} w^\eta(X_\rho^{\tau_1,\xi_1,\alpha^1})\right)\nonumber\\
&\le I_D\int_\tau^{\tau_1} e^{-\beta s} f(X_s^{\tau,\xi,\overline\alpha},\overline\alpha_s)\,ds+I_D Z^{\tau_1,\xi_1,\alpha^1}_\rho(w).
\end{align}
By Definition \ref{def:1} we get
\begin{align} \label{eq:2.5}
\mathsf E(Z^{\tau_1,\xi_1,\alpha^1}_\rho(w) I_D|\mathcal F_{\tau_1}) &=\mathsf E(Z^{\tau_1,\xi_1,\alpha^1}_{\rho_D}(w) I_D|\mathcal F_{\tau_1})\le I_D e^{-\beta\tau_1}w(\xi_1)\nonumber\\
&= I_D e^{-\beta\tau_1}w^\eta(\xi_1).
\end{align}
The last equality follows from the fact that $\xi_1\not\in  B_{\varepsilon/2}(x_0)$ on the set $\{\rho>\tau_1\}$ and $w=w^\eta$ on $G\backslash B_{\varepsilon/2}(x_0)$. From (\ref{eq:2.4}), (\ref{eq:2.5}) it follows that
\begin{align*}
\mathsf E(Z^{\tau,\xi,\alpha}_\rho(w^\eta) I_D|\mathcal F_{\tau_1})
&\le I_D\left(\int_\tau^{\tau_1} e^{-\beta s} f(X_s^{\tau,\xi,\overline\alpha},\overline\alpha_s)\,ds+e^{-\beta\tau_1}w^\eta(\xi_1)\right)\nonumber\\
&= I_D Z^{\tau,\xi,\overline\alpha}_{\tau_1}(w^\eta),
\end{align*}
and we obtain the estimate
\begin{align} \label{eq:2.6}
\mathsf E(Z^{\tau,\xi,\alpha}_\rho(w^\eta)|\mathcal F_{\tau}) &=\mathsf E(I_{\{\rho\le\tau_1\}}Z^{\tau,\xi,\alpha}_\rho(w^\eta)|\mathcal F_{\tau})+\mathsf E(I_{\{\rho>\tau_1\}}\mathsf E(Z^{\tau,\xi,\alpha}_\rho(w^\eta)|\mathcal F_{\tau_1})|\mathcal F_\tau)\nonumber\\
& \le \mathsf E(I_{\{\rho\le\tau_1\}}Z^{\tau,\xi,\overline\alpha}_\rho(w^\eta)|\mathcal F_{\tau})+\mathsf E(I_{\{\rho>\tau_1\}} Z^{\tau,\xi,\overline\alpha}_{\tau_1}(w^\eta)|\mathcal F_{\tau})\nonumber\\
& =\mathsf E(Z^{\tau,\xi,\overline\alpha}_{\rho\wedge\tau_1}(w^\eta)|\mathcal F_\tau).
\end{align}

On the stochastic interval $\llbracket\tau_H,(\tau_1)_H\rrbracket$ the trajectories of $X^{\tau,\xi,\overline\alpha}$ do not leave the ball $B_{\varepsilon/2}(x_0)$. Hence, the estimate $w^\eta(X^{\tau,\xi,\overline\alpha}_{\rho\wedge\tau_1})\le\varphi^\eta(X^{\tau,\xi,\overline\alpha}_{\rho\wedge\tau_1})$ holds true on $H$ and we get the inequality
\begin{equation} \label{eq:2.7}
Z^{\tau,\xi,\overline\alpha}_{\rho\wedge\tau_1}(w^\eta)=Z^{\tau,\xi,a}_{\rho\wedge\tau_1}(w^\eta) I_H+Z^{\tau,\xi,\alpha^0}_{\rho\wedge\tau_1}(w^\eta) I_{H^c}\le Z^{\tau,\xi,a}_{\rho\wedge\tau_1}(\varphi^\eta) I_H+Z^{\tau,\xi,\alpha^0}_{\rho\wedge\tau_1}(w) I_{H^c}.
\end{equation}

Applying Ito's formula
\begin{align} \label{eq:2.8}
Z^{\tau,\xi,a}_t(\varphi^\eta) &=\int_\tau^t e^{-\beta s} f(X_s^{\tau,\xi,a},a)\,ds+e^{-\beta t}\varphi^\eta(X^{\tau,\xi,a}_t) \nonumber\\
&=e^{-\beta\tau}\varphi^\eta(\xi)+\int_\tau^t e^{-\beta s}\left[f(X_s^{\tau,\xi,a},a)+(\mathcal L^a\varphi^\eta-\beta\varphi^\eta)(X^{\tau,\xi,a}_s)\right]\,ds \nonumber\\
&+\int_\tau^t e^{-\beta s}\varphi^\eta_x(X_s^{\tau,\xi,a})\cdot\sigma(X_s^{\tau,\xi,a},a)\,dW_s.
\end{align}
on the interval $\llbracket\tau,\rho\wedge\tau_1\rrbracket$, taking the conditional expectation, and using (\ref{eq:2.3}), we get
\begin{equation} \label{eq:2.9}
\mathsf E(Z^{\tau,\xi,a}_{\rho\wedge\tau_1}(\varphi^\eta) I_H|\mathscr F_\tau)\le e^{-\beta\tau}\varphi^\eta(\xi) I_H=e^{-\beta\tau} w^\eta(\xi) I_H=Z^{\tau,\xi,\alpha}_\tau(w^\eta) I_H.\end{equation}
Furthermore,
\begin{equation} \label{eq:2.10}
\mathsf E(Z^{\tau,\xi,\alpha^0}_{\rho\wedge\tau_1}(w) |\mathscr F_\tau)I_{H^c} \le Z_\tau^{\tau,\xi,\alpha^0}(w) I_{H^c}=Z_\tau^{\tau,\xi,\alpha}(w^\eta) I_{H^c}
\end{equation}
by the definition of $\alpha^0$. The combination of (\ref{eq:2.9}), (\ref{eq:2.10}) with (\ref{eq:2.7}) and (\ref{eq:2.6}) gives the desired inequality
\[
\mathsf E(Z^{\tau,\xi,\alpha}_\rho(w^\eta) |\mathscr F_\tau)\le Z_\tau^{\tau,\xi,\alpha}(w^\eta).  \qedhere
\]
\end{proof}

To show that $w_+$ satisfies the last assertion of Theorem \ref{th:1}, we study its behavior near the points of $\Gamma$. Fix $x\in\Gamma$. By the definition of $\Gamma$ there exists $\alpha^1\in\mathscr A(x)$ such that
\begin{equation} \label{eq:2.11}
\tau=\inf\{t>0:X_t^{x,\alpha^1}\in G^\circ\}=0\ \ a.s.
\end{equation}

For $\varepsilon>0$ consider the predictable set
\[ E=\{(t,\omega): X_t^{x,\alpha^1}(\omega)\in G^\circ,\ t\in (0,\varepsilon]\}=\rrbracket 0,\varepsilon\rrbracket \cap\bigl(X^{x,\alpha^1}\bigr)^{-1}(G^\circ)\]
and its projection: $D=\{\omega:(t,\omega)\in E \quad\text{for some}\quad t\in[0,\infty)\}$. The equality (\ref{eq:2.11}) means that $\mathsf P(D)=1$. By the section theorem \cite[Theorem 16.12]{Bass11} there exist an $\mathbb F$-stopping time $\sigma^\varepsilon$ such that
\begin{equation} \label{eq:2.12}
\{(\sigma^\varepsilon(\omega),\omega):\omega\in\Omega, \sigma^\varepsilon(\omega)<\infty\}\subset E,\ \ \ \mathsf P(\sigma^\varepsilon<\infty)\ge 1-\varepsilon.
\end{equation}
Put $D_\varepsilon=\{\sigma^\varepsilon\le \varepsilon\}=\{\sigma^\varepsilon<\infty\}$. Then (\ref{eq:2.12}) means that
\[ X^{x,\alpha^1}_{\sigma^\varepsilon}\in G^\circ \quad\text{on}\quad D_\varepsilon,\ \ \
\mathsf P(D_\varepsilon)\ge 1-\varepsilon.\]

Let $w$ be a stochastic supersolution, bounded from above by the constant $\overline f/\beta$. Put $\xi^\varepsilon=I_{D_\varepsilon} X^{x,\alpha^1}_{\sigma^\varepsilon}\in G^\circ$ and take a $w$-suitable control $\alpha^2\in\mathscr A(\sigma^\varepsilon,\xi^\varepsilon)$. Then
\[\alpha=\alpha^1 I_{\{t<\sigma^\varepsilon\}}+\alpha^2 I_{\{t\ge\sigma^\varepsilon\}}\in\mathscr A(x).\]
Taking into account that $\sigma^\varepsilon=\infty$ on $D_\varepsilon^c$, by the definitions of $v$ and $w$ we obtain:
\begin{align*}
 v(x) & \le\mathsf E\left(\int_0^{\sigma^\varepsilon} e^{-\beta t} f(X_t^{x,\alpha^1},\alpha_t^1)\,dt  + \mathsf E\left(\int_{\sigma^\varepsilon}^\infty e^{-\beta t} f(X_t^{\sigma^\varepsilon, \xi^\varepsilon,\alpha^2},\alpha_t^2)\,dt\Bigr|\mathscr F_{\sigma^\varepsilon}\right) \right),\\
 &\le\mathsf E\left(\int_0^{\sigma^\varepsilon} e^{-\beta t} f(X_t^{x,\alpha^1},\alpha_t^1)\,dt  +e^{-\beta\sigma^\varepsilon} w(\xi^\varepsilon)\right)
\end{align*}
It easily follows that
\begin{equation} \label{eq:2.13}
v(x) \le \frac{\overline f}{\beta}\left(1-\mathsf E e^{-\beta\sigma^\varepsilon}\right)+\mathsf E e^{-\beta\sigma^\varepsilon} w(\xi^\varepsilon) I_{D_\varepsilon}+\frac{\overline f}{\beta}(1-\mathsf P(D_\varepsilon)).
\end{equation}
Moreover, by Lemma \ref{lem:3} and the monotone convergence theorem we can change $w$ to $w_+$ in this inequality.

Take $\varepsilon_n$ such that $\mathsf P(D_{\varepsilon_n}^c)\le 1/2^n$. By the Borel-Cantelli lemma for all $\omega$ in some set $\Omega'$ with $\mathsf P(\Omega')=1$ we have $\omega\in D_{\varepsilon_n}$ for sufficiently large $n$. Thus,
\[ I_{D_{\varepsilon_n}}\to 1, \ \ \xi^{\varepsilon_n}\to x,\ \ \sigma^{\varepsilon_n}\to 0 \quad\text{on}\quad \Omega',\]
and from (\ref{eq:2.13}) we obtain the estimate
$ v(x)\le \limsup_{G^\circ\ni y\to x} w_+(y).$

\section{Stochastic subsolutions}
\label{sec:3}
\setcounter{equation}{0}
\begin{definition} \label{def:2}
With the notation of Section \ref{sec:2} we call $u\in C_b(G)$ a \emph{stochastic subsolution} if
\begin{equation} \label{eq:3.1}
\mathsf E(Z^{\tau,\xi,\alpha}_\rho(u)|\mathscr F_\tau)\ge Z^{\tau,\xi,\alpha}_\tau(u)=e^{-\beta\tau}u(\xi)
\end{equation}
for any randomized initial condition $(\tau,\xi)$, admissible control process $\alpha\in\mathscr A(\tau,\xi)$ and stopping time $\rho\ge\tau$.
\end{definition}

Any stochastic subsolution $u$ is a lower bound for $v$: for $\tau=0$, $\xi=x$, $\rho=\infty$ we have
\[ J(x,\alpha)=\mathsf E Z_\infty^{x,\alpha}(u)\ge Z_0^{x,\alpha}(u)=u(x),\ \ \ \alpha\in\mathscr A(x).\]

Put $\underline f=\inf_{(x,a)\in G\times A} f(x,a).$ The set $\mathcal V^-$ of stochastic subsolutions is non-empty and contains sufficiently large negative constants $c$. Indeed, it is easy to see that
\[ \mathsf E(Z_\rho^{\tau,\xi,\alpha}(c)|\mathscr F_\tau)\ge c e^{-\beta\tau} \quad\text{for}\quad c\le\underline f/\beta.\]

\begin{lemma} \label{lem:5}
Let $u_1$, $u_2$ be stochastic subsolutions. Then $u_1\vee u_2$ is a stochastic subsolution.
\end{lemma}

The proof follows from the inequality
\[\mathsf E(Z^{\tau,\xi,\alpha}_\rho(u_1\vee u_2)|\mathcal F_\tau) \ge \max_{i=1,2}\mathsf E(Z^{\tau,\xi,\alpha}_\rho(u_i)|\mathcal F_\tau)\ge\max_{i=1,2}Z^{\tau,\xi,\alpha}_\tau(u_i)= e^{-\beta\tau}(u_1\vee u_2)(\xi).\]
\begin{lemma} \label{lem:6}
There exists a sequence $u_n\in\mathcal V^-$, $u_n(x)\le u_{n+1}(x)$, $x\in G$ such that
\[\lim_{n\to\infty} u_n(x)=u_-(x):=\sup\limits_{u\in\mathcal V^-} u(x).\]
\end{lemma}
This lemma is analogous to Lemma \ref{lem:3}.

\begin{lemma} \label{lem:7}
The function
\[ u_-(x)=\sup\limits_{u\in\mathcal V^-} u(x)\]
is a viscosity supersolution of (\ref{eq:1.7}).
\end{lemma}
\begin{proof}
If $u_-$ is not a viscosity supersolution then there exist $x_0\in G$, $\varphi\in C^2$ and $\varepsilon>0$ such that $u_-(x_0)=\varphi(x_0)$, $u_->\varphi$ on $(\overline{B_\varepsilon(x_0)}\backslash\{0\})\cap G$ and
\[ F(x_0,\varphi(x_0),D\varphi(x_0),D^2\varphi(x_0))<0.\]
By the continuity of $F$ we can assume that
\begin{equation} \label{eq:3.2}
 F(x,\varphi(x),D\varphi(x),D^2\varphi(x))<0,\ \ x\in B_\varepsilon(x_0)\cap G.
\end{equation}
Furthermore, by the lower-semicontinuity of $u_-$ we have
\[ u_-(x)\ge\varphi(x)+\delta,\ \ x\in S_\varepsilon:=\left(\overline B_\varepsilon(x_0)\backslash B_{\varepsilon/2}(x_0)\right)\cap G\]
for some $\delta>0$. In the same way as in the proof of Lemma \ref{lem:4}, one can show that there exist $u\in\mathcal V^-$ and $\delta'\in (0,\delta)$ such that $u\ge\varphi+\delta'$ on $S_\varepsilon$.

Take an $\eta\in (0,\delta')$ such that (\ref{eq:3.2}) holds true for $\varphi^\eta=\varphi+\eta$ instead of $\varphi$. We have $u-\varphi^\eta\ge \delta'-\eta>0$ on $S_\varepsilon$.

To get a contradiction it is enough to prove that the function
\[ u^\eta=\left\{\begin{array}{l l}
    \varphi^\eta\vee u  & \text{on}\quad B_\varepsilon(x_0)\cap G,\\
    u & \ \text{otherwise}
  \end{array} \right. \]
is a stochastic subsolution, since $u^\eta(x_0)=\varphi^\eta(x_0)> u_-(x_0)$, contrary to the definition of $u_-$.

Clearly $u^\eta\in C_b(G)$, and we only should to verify (\ref{eq:3.1}) for any randomized initial condition $(\tau,\xi)$, control process $\alpha\in\mathscr A(\tau,\xi)$ and stopping time $\rho\ge \tau$. Put
\[\tau_1 = \inf\{t\ge\tau:X^{\tau,\xi,\alpha}_t\not\in B_{\varepsilon/2}(x_0)\},\ \ \xi_1=X^{\tau,\xi,\alpha}_{\tau_1} I_{\{\tau_1<\infty\}},\ E=\{\xi\in B_{\varepsilon/2}(x_0)\}.\]
We have
\[ \xi_1\in\partial B_{\varepsilon/2}(x_0)\cap G \quad\text{on}\quad E\cap\{\tau_1<\infty\};\quad \xi_1=0 \quad\text{on}\quad E\cap\{\tau_1=\infty\};\quad \xi_1=\xi \quad\text{on}\quad E^c.
\]
Moreover, $X^{\tau_1,\xi_1,\alpha}=X^{\tau,\xi,\alpha}$ on the stochastic interval $\llbracket\tau_1,\infty\rrbracket$.

Put $D=\{\rho>\tau_1\}$. Similarly to (\ref{eq:2.4}) we get
\begin{align} \label{eq:4.4}
Z^{\tau,\xi,\alpha}_\rho(u^\eta) I_D \ge I_D\int_\tau^{\tau_1} e^{-\beta s} f(X_s^{\tau,\xi,\alpha},\alpha_s)\,ds+I_D Z^{\tau_1,\xi_1,\alpha}_\rho(u).
\end{align}
Applying Definition \ref{def:2}, we obtain
\begin{equation} \label{eq:4.5}
\mathsf E(Z^{\tau_1,\xi_1,\alpha}_\rho(u) I_D|\mathcal F_{\tau_1}) =\mathsf E(Z^{\tau_1,\xi_1,\alpha}_{\rho_D}(u) I_D|\mathcal F_{\tau_1})\ge I_D e^{-\beta\tau_1} u(\xi_1)=I_D e^{-\beta\tau_1}u^\eta(\xi_1).
\end{equation}
The last equality follows from the fact that $\xi_1$, restricted to $D$, takes values in the set $G\backslash B_{\varepsilon/2}(x_0)$ where $u=u^\eta$.

From (\ref{eq:4.4}), (\ref{eq:4.5}) it follows that
\begin{align} \label{eq:3.5}
\mathsf E(Z^{\tau,\xi,\alpha}_\rho(u^\eta) I_D|\mathcal F_{\tau_1})
&\ge I_D\left(\int_\tau^{\tau_1} e^{-\beta s} f(X_s^{\tau,\xi,\alpha},\alpha_s)\,ds+e^{-\beta\tau_1}u^\eta(\xi_1)\right)\nonumber\\
&= I_D Z^{\tau,\xi,\alpha}_{\tau_1}(u^\eta).
\end{align}
By (\ref{eq:3.5}) we have
\begin{align} \label{eq:3.6}
\mathsf E(Z^{\tau,\xi,\alpha}_\rho(u^\eta)|\mathcal F_{\tau}) &=\mathsf E(I_{\{\rho\le\tau_1\}}Z^{\tau,\xi,\alpha}_\rho(u^\eta)|\mathcal F_{\tau})+\mathsf E(I_{\{\rho>\tau_1\}}\mathsf E(Z^{\tau,\xi,\alpha}_\rho(u^\eta)|\mathcal F_{\tau_1})|\mathcal F_\tau)\nonumber\\
& \ge \mathsf E(I_{\{\rho\le\tau_1\}}Z^{\tau,\xi,\alpha}_\rho(u^\eta)|\mathcal F_{\tau})+\mathsf E(I_{\{\rho>\tau_1\}} Z^{\tau,\xi,\alpha}_{\tau_1}(u^\eta)|\mathcal F_{\tau})\nonumber\\
& =\mathsf E(Z^{\tau,\xi,\alpha}_{\rho\wedge\tau_1}(u^\eta)|\mathcal F_\tau).
\end{align}

Put
\[ U=\{x\in G\cap B_{\varepsilon/2}(x_0):\varphi^\eta(x)>u(x)\},\ \ H=\{\xi\in U\}\in\mathscr F_\tau.\]
On the stochastic interval $\llbracket\tau_H,(\rho\wedge\tau_1)_H\rrbracket$ the trajectories of $X^{\tau,\xi,\alpha}$ do not leave the set $B_{\varepsilon/2}(x_0)\cap G$. Hence, we have $u^\eta(X^{\tau,\xi,\alpha}_{\rho\wedge\tau_1})I_H\ge\varphi^\eta(X^{\tau,\xi,\alpha}_{\rho\wedge\tau_1})I_H$ and
\begin{equation} \label{eq:3.7}
Z^{\tau,\xi,\alpha}_{\rho\wedge\tau_1}(u^\eta)\ge Z^{\tau,\xi,\alpha}_{\rho\wedge\tau_1}(\varphi^\eta) I_H+Z^{\tau,\xi,\alpha}_{\rho\wedge\tau_1}(u) I_{H^c}.
\end{equation}

Apply Ito's formula (\ref{eq:2.8}) on the interval $\llbracket\tau,\rho\wedge\tau_1\rrbracket$ with $\alpha$ instead of $a$. Taking the conditional expectation and using (\ref{eq:3.2}), we get
\begin{equation} \label{eq:3.8}
 \mathsf E(Z^{\tau,\xi,\alpha}_{\rho\wedge\tau_1}(\varphi^\eta) I_H|\mathscr F_\tau)\ge e^{-\beta\tau}\varphi^\eta(\xi) I_H=e^{-\beta\tau} u^\eta(\xi) I_H=Z^{\tau,\xi,\alpha}_\tau(u^\eta) I_H.
\end{equation}
Furthermore,
\begin{equation} \label{eq:3.9}
 \mathsf E(Z^{\tau,\xi,\alpha}_{\rho\wedge\tau_1}(u) |\mathscr F_\tau)I_{H^c} \ge Z_\tau^{\tau,\xi,\alpha}(u) I_{H^c}=Z_\tau^{\tau,\xi,\alpha}(u^\eta) I_{H^c},
\end{equation}
and the desired inequality
\[ \mathsf E(Z^{\tau,\xi,\alpha}_\rho(u^\eta) |\mathscr F_\tau)\ge Z_\tau^{\tau,\xi,\alpha}(u^\eta)\]
follows from (\ref{eq:3.8}), (\ref{eq:3.9}), combined with (\ref{eq:3.6}), (\ref{eq:3.7}).
\end{proof}

\section{The case of a smooth domain}
\label{sec:4}
\setcounter{equation}{0}
Let $G$ coincide with the closure of $G^\circ$, and assume that $\partial G$ is of class $C^2$. Then the distance function $\rho$ from $\partial G$:
\[ \rho(x)=\inf\{y\in G^c:|y-x|\},\quad x\in G \]
is of class $C^2$ in a neighbourhood of $\partial G$ (see \cite[Lemma 14.16]{GilTru01}). Put $-n(x)=D\rho(x)$, $x\in G$. If $x\in\partial G$, $n(x)$ is the unit outer normal to $\partial G$ at $x$. It is shown in \cite[Example 3.2]{BarGoa99}, \cite[Example 1]{BarJen02} that condition (\ref{eq:1.3}) is reduced to the following: for any $x\in\partial G$ there exists $a\in A$ such that
\[ \sigma^T(x,a)n(x)=0,\ \  -n(x)\cdot b(x,a)+\frac{1}{2}\Tr\left(\sigma(x,a)\sigma^T(x,a)D^2 \rho(x)\right)\ge 0.\]

To get a comparison result we need a stronger condition, presented in the next theorem.

\begin{theorem} \label{th:2}
Assume that there exists a Borel measurable function $\psi:G\mapsto A$ such that the functions (\ref{eq:1.4})
are globally Lipschitz continuous and
\begin{equation} \label{eq:4.1}
 \sigma_\psi(x)=0,\ \ -n(x)\cdot b_\psi(x)>0,\ \ x\in\partial G.
\end{equation}
Then the value function $v$, defined by (\ref{eq:1.2}), is the unique continuous viscosity solution of the state constrained problem (\ref{eq:1.7}).
\end{theorem}
\begin{proof}
The viscosity subsolution $w_+$, specified in Theorem \ref{th:1}, satisfies also the linear inequality
\begin{equation} \label{eq:4.2}
\beta w_+(x)-f(x,\psi(x))-(b_\psi\cdot Dw_+)(x)-\frac{1}{2}\Tr(\sigma_\psi\sigma^T_\psi D^2 w_+)(x)\le 0,\ \ x\in G^\circ
\end{equation}
in the viscosity sense. Consider the function
\[\widetilde w_+(x)=\left\{\begin{array}{l l}
    \limsup\limits_{G^\circ\ni y\to x} w_+(y) & x\in\partial G,\\
    w_+(x) & \ \text{otherwise}.
  \end{array} \right. \]
Clearly, $\widetilde w_+$ is a viscosity subsolution of (\ref{eq:4.2}), satisfying all conditions of Theorem \ref{th:1}.

Now we use conditions (\ref{eq:4.1}). By Lemma 4.1 of \cite{BarRou98} the function $\widetilde w_+$ is a viscosity subsolution of (\ref{eq:4.2}) on $G$. Furthermore, by Theorem 4.1(ii) of \cite{BarRou98}, for any $x\in\partial G$ there exists a sequence $x_k\in G^\circ$, $x_k\to x$ such that
$\widetilde w_+(x)=\lim_{k\to\infty} \widetilde w_+(x_k)$ and
\[\limsup_{k\to\infty}\frac{|x_k-x|}{d(x_k)}<\infty,\]
or, equivalently,
\[ \limsup_{k\to\infty}\frac{(x_k-x)\cdot n(x)}{|x_k-x|}\le -\beta\]
for some $\beta\in (0,1)$. This is the \emph{nontangential upper semicontinuity} property of $\widetilde w_+$, which, by the comparison result of \cite{Kats94} (Theorem 2.2), implies that
\begin{equation} \label{eq:4.3}
\widetilde w_+\le u_-\quad \text{on } G.
\end{equation}

Let us prove that $\partial G=\Gamma$. For $x\in\partial G$ denote by $\overline X$ the solution of the equation
\[ X_t=x+\int_0^t b_\psi(X_s)\,ds+\int_0^t \sigma_\psi(X_s)\, dW_s,\ \ x\in\partial G. \]
Since conditions (\ref{eq:4.1}) imply the viability, we get an admissible control $\alpha_t=\psi(\overline X_t)$: $\overline X_t=X_t^{x,\alpha}\in G$, $t\ge 0$ a.s. Take $\varepsilon>0$ such that $\rho\in C^2(B_\varepsilon(x))$ and
\[ \inf_{y\in B_\varepsilon(x)\cap G}\left[-n(y)\cdot b_\psi(y)+\frac{1}{2}\Tr\left(\sigma_\psi(y)\sigma_\psi^T(y)D^2 \rho(y)\right)\right]>0.  \]
Furthermore, put $\tau=\inf\{t\ge 0: \overline X_t\not\in B_\varepsilon(x)\}$. By Ito's formula we have
\[\rho(\overline X_{t\wedge\tau})=\rho(x)-\int_0^{t\wedge\tau} n(\overline X_s)\cdot b_\psi(\overline X_s)\,ds +
\frac{1}{2}\int_0^{t\wedge\tau}\Tr\left(\sigma_\psi(\overline X_s)\sigma_\psi^T(\overline X_s)D^2\rho(\overline X_s)\right)\,ds+M_t, \]
where $M$ is a continuous martingale with $M_0=0$. From the representation of $M$ as a time-changed Brownian motion on an extended filtered probability space (see \cite[Theorem 7.2']{IkeWat89}) it follows that $0$ is a limit point of the set $\{t>0:M_t=0\}$ a.s. For a sequence $t_k(\omega)\to 0$ with $M_{t_k}=0$ we have
\[ \rho(\overline X_{t_k}) =\rho(x)+\int_0^{t_k} \left[-n(\overline X_s)\cdot b_\psi(\overline X_s)\,ds +
\frac{1}{2}\Tr\left(\sigma_\psi(\overline X_s)\sigma_\psi^T(\overline X_s)D^2\rho(\overline X_s)\right)\right]\,ds>0\quad a.s. \]
for sufficiently large $k$. Thus, $\overline X$ immediately enters $G^\circ$:
$$ \inf\{t>0:\overline X_t\in G^\circ\}=\inf\{t>0:\rho(\overline X_t)>0\}=0\ \ a.s.,$$
and we conclude that $x\in\Gamma$ and $\partial G=\Gamma$.

This fact, together with Theorem \ref{th:1} and inequality (\ref{eq:4.3}), implies that
\[ v\le\widetilde w_+\le u_-\le v \quad\text{on}\quad G.\]
Hence, $v=\widetilde w_+=u_-$ is a continuous function, and it satisfies (\ref{eq:1.7}) in the viscosity sense. Note also that the uniqueness of a continuous constrained viscosity solution is a more classical result: see \cite[Theorem 7.10]{CraIshLio92}.
\end{proof}

Theorem \ref{th:2} is similar to Theorem 4.1 of \cite{Kats94}. Although, the second condition (\ref{eq:4.1}) is presented there in the form
\[ -n(x)\cdot b_\psi(x)+\frac{1}{2}\Tr\left(\sigma_\psi(x)\sigma_\psi^T(x)D^2 \rho(x)\right)\ge c>0,\quad x\in\partial G,\]
which is formally not comparable to ours local condition $-n(x)\cdot b_\psi(x)>0$, $x\in\partial G$, the result of \cite{Kats94}
is more sophisticated. To get the comparison result in Theorem \ref{th:2} we used only the fact that any subsolution, being suitably modified at the boundary points, possesses the nontangential upper semicontinuity property under conditions (\ref{eq:4.1}). In \cite{Kats94} it is shown that a subsolution $u\ge v$ with this property exists even some diffusion in the tangent direction to $\partial G$ is allowed: see conditions A3 of \cite{Kats94}.

Certainly, the stochastic Perron method can be applied in the case of finite horizon as well.  However, some work is required
to study the parabolic problem, corresponding to (1.7). In particular, a new boundary condition at the terminal
time appears, and the viability notion should be modified. Such a problem was studied in \cite{BouNut12} by another methods. We mention a comparison result, ensuring the continuity of the value function, proved under conditions similar to (4.1):
see \cite[Theorem A.1]{BouNut12}.


\end{document}